\documentclass[11pt]{amsart}
\usepackage{amssymb, amsmath}
\usepackage{amsthm}
\usepackage{amscd}
\usepackage{graphicx}

\newtheorem{theorem}{Theorem}
\newtheorem{lemma}[theorem]{Lemma}
\newtheorem*{theoremA}{Theorem A}
\newtheorem*{theoremA'}{Theorem A'}

\newtheorem*{lemmaB}{Lemma B}

\theoremstyle{definition}

\newtheorem*{question}{Question}

\newtheorem*{remark}{Remark}
\newtheorem*{acknowledgement}{Acknowledgement}

\title{Non-separating immersions of spheres and Bing houses}

\author{Michael Freedman}
\address{\hskip-\parindent
Michael Freedman\\
    Microsoft Research, Station Q, and Department of Mathematics \\
    University of California, Santa Barbara \\
    Santa Barbara, CA 93106 
}
\author{T. T$\hat{\mathrm{a}}$m Nguy$\tilde{\hat{\mathrm{e}}}$n-Phan}
\address{\hskip-\parindent
T. Tam Nguyen-Phan\\
Max Planck Institut f\"ur Mathematik\\
Bonn\\
Germany}
\email{tam@mpim-bonn.mpg.de}

\def\ra{\rightarrow}

\def\beqa{\begin{eqnarray}}
\def\eeqa{\end{eqnarray}}
\def\beqa{\begin{eqnarray}}
\def\eeqa{\end{eqnarray}}

\def\R{\mathbb{R}}

\input xy
\xyoption{all}

\begin{document}

\maketitle

\section{Introduction}
It is a common topology homework exercise to prove that any embedding of an orientable $n$-manifold $M$, such as the $n$-sphere $\mathbb{S}^{n}$, into $\mathbb{R}^{n+1}$ is separating\footnote{That is, the complement of the image of the embedding has at least two components.}. However, we have never seen this homework problem posed with ``embedding" replaced by ``immersion". At first glance it only seems to  complicate the problem for no good reason. But the truth is subtle, with topological and PL immersions behaving differently from their smooth cousins. Topological and PL immersions do come up in ``nature" so one cannot avoid them, and when one does not avoid them one discovers some interesting topology. This paper is about immersions of spheres but we will open with a geometry tale.

\subsection*{A geometry tale}
One place to look for an immersion of an $n$-sphere is at the \emph{cut locus}\footnote{The \emph{cut locus} of $p$ is the set of \emph{cut points} of $p$ along all the geodesics starting from $p$. A \emph{cut point} of $p$ along a geodesic ray $\gamma_v$ with initial velocity $v$ is the point $\gamma_v (t_0)$, where $t_0$ is the first time after which $\gamma_v([0,t])$ stops being the minimizing geodesic from $p$ to $\gamma_v (t)$.} of a point $p$ in some closed, Riemannian manifolds $M^{n}$. Not any Riemannian manifold, but those that have small diameter and small upper bound on the sectional curvatures. In such circumstances, the cut locus of $p$ is the image of an (topological) immersion of a $(n-1)$-sphere whose complement is connected. We give an short explanation of this in the next paragraph.

Recall that if the sectional curvature of $M$ is bounded from above by some number $\kappa >0$, then by comparing to the round sphere of radius $(1/\sqrt{\kappa})$ one deduces that the exponential map $\exp_p$ at $p\in M$ is a local diffeomorphism on the ball of radius $(\pi/\sqrt{\kappa})$. Suppose (for simplicity) that $M$ has both the curvature and diameter $\leq 1$.
Then for each unit tangent vector $v \in T_p(M)$, if one follows the unit-speed geodesic ray $\gamma_v$ starting at $p$ in the direction of $v$, one reaches at some time $t=t_v \leq 1$ a \emph{cut point} $f(v) =\gamma (t_v)$ along $\gamma$. Letting $v$ vary over the unit sphere in $T_p(M)$ and taking into account $t_v$, one obtains a  \emph{star-shaped}\footnote{``Star-shaped" means a sphere about the origin of $\mathbb{R}^{n}$ which in polar coordinates has the form $\{(\Omega, \rho(\Omega) )\}$ for some continuous function $\rho\colon \{\text{angles}\} \rightarrow\mathbb{R}^+$. Such embedded spheres have bicollared neighborhoods and thus are topologically flat.} \emph{locally flat} embedding $\mathbb{S}^{n-1}$ into $T_pM$, namely the set $ C: = \{t_v v\; |\; v \in T_p^1(M)\}$, which bounds a topologically flat ball $B$ on which the map $\exp_p$ is injective. Since $\exp_p$ is a local diffeomorphism on a ball of radius $\pi$, when restricted to $C =\partial B = \mathbb{S}^{n-1}$ it gives a topologically flat immersion  
\[f\colon \mathbb{S}^{n-1} \rightarrow M\] 
whose image is exactly the cut locus of $p$, the complement of which is $\exp_p(\text{int}(B))$, which is diffeomorphic to an open ball and thus is connected. 

This was how Buser and Gromoll (\cite{BuserGromoll}) showed that $\mathbb{S}^2$ does not admit a metric with small curvature\footnote{By ``small curvature" we mean that the sectional curvatures have a small upper bound. The lower bound of the sectional curvature can be very negative. } and small diameter. If $\mathbb{S}^2$ admits such a metric, then one could use this metric to obtain an non-separating immersion of $\mathbb{S}^1$ into $\mathbb{S}^2$, which, as a consequence of the Jordan Curve Theorem, is impossible. 

In the same paper, Buser and Gromoll explained how Gromov constructed metrics on $\mathbb{S}^3$ with upper bounds on the  curvature $\kappa$ and diameters approaching zero. Closed manifolds that admit such metrics are called \emph{almost negatively curved}. It is interesting that $\mathbb{S}^3$ is almost negatively curved, but for our purpose this implies there is a non-separating immersion of $\mathbb{S}^2$ into $\mathbb{S}^3$, in fact one with contractible image. Thus the modification of the homework problem above has a surprising answer (in topological and PL categories): It would be a ``trick question". So we ask,

\begin{question}
Are there PL immersions of $\mathbb{S}^{n-1}$ into $\mathbb{R}^{n}$ , for each $n\geq 3$, that have contractible image and in particular are non-separating? If so, how can we obtain them?
\end{question}

The metrics on $\mathbb{S}^3$ described in \cite{BuserGromoll} are obtained by modifying the round metric by installing a number of ``short-cuts" that make the diameter small without changing the topology or increasing the curvature. These ``short-cuts" are not at all like the worm holes of a science-fiction, which are easy plot device installed at the expense of changing the topology. They are much more delicate in that they preserve the topology of the manifold. The resulting metrics are stranger than fiction and there is little hope in extracting an explicit immersion of $\mathbb{S}^2$ from them. However, back in the 1950's, such immersions of $\mathbb{S}^2$ had already been discovered, apparently by R. H. Bing.

\subsection*{The house that Bing built}
Topologists are often introduced to the idea of a \emph{contractible}, but not \emph{collapsible}, complex through the example of ``the house with two rooms", also called ``the house that Bing built", which is the 2-complex $X\subset \mathbb{R}^3$ as drawn in Figure 1. 

\begin{figure}[h!]
\label{2d-bing}
\centering
\includegraphics[scale=0.43]{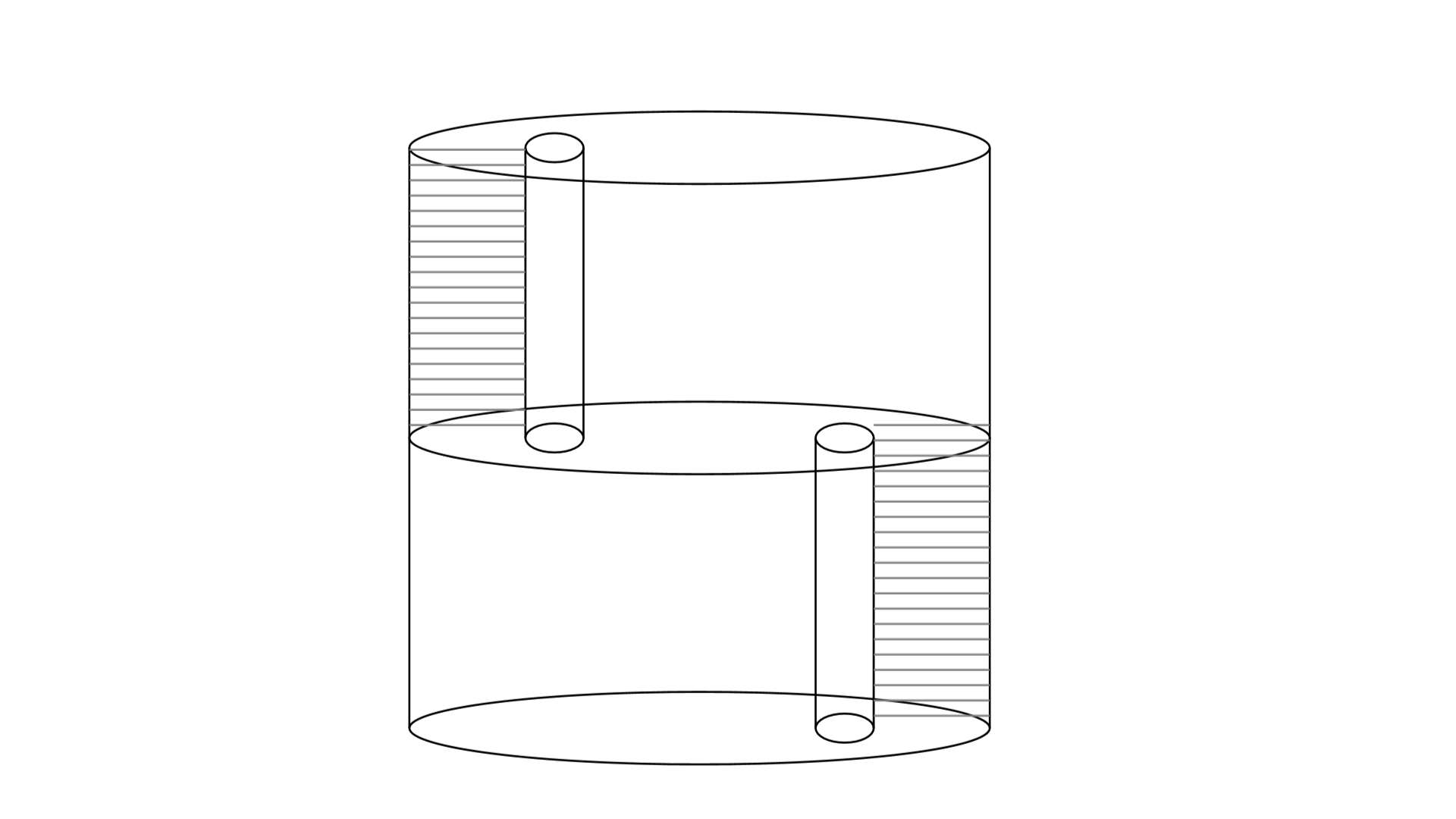}
\caption{The house with two rooms.}
\end{figure}

This 2-complex $X$ has a regular neighborhood $\mathcal{N}(X)$ that is P.L. homeomorphic to $\mathbb{D}^3$ and possesses a mapping cylinder structure $\mathcal{N}(X) \cong \partial\mathcal{N}(X) \times [0,1] / n \times 1 \equiv n' \times 1$ iff $f(n) = f(n')$, where the map $f\colon\partial \mathcal{N}(X) \rightarrow X$ is required to be a PL immersion (of $\mathbb{S}^2$). In particular, pushing along the mapping cylinder lines yields a strong deformation retraction $r_t: \mathcal{N}(X) \ra X$, $t \in [0,1]$, with $r_1$ the retraction. To see the non-separating immersion $f\colon \mathbb{S}^2\rightarrow \mathbb{R}^3$, imagine shrink-wrapping Figure 1 with a very flexible spherical membrane that will enter the long tubes and then stretch to fill the chambers to which the tubes lead.

In higher dimension we call an $(n-1)$-complex $Y^{n-1} \subset \mathrm{int}(\mathbb{D}^n)$ a \emph{Bing house} if its PL regular neighborhood $\mathcal{N}(Y^{n-1})$ is a PL $n$-ball. This implies the existence of a collapse $\mathbb{D}^n \searrow Y^{n-1}$ and, of course, a strong deformation retract $r_t \colon \mathcal{N}(Y) \ra Y$. As a consequence, $\mathcal{N}(Y^{n-1})$ possesses a mapping cylinder structure induced by some PL immersion $f\colon \partial\mathcal{N}(Y^{n-1}) \rightarrow Y^{n-1}$. Clearly, a $(n-1)$-dimensional Bing house gives a non-separating immersion $\mathbb{S}^{n-1} =  \partial\mathcal{N}(Y^{n-1})\rightarrow\mathbb{R}^n$.

\subsection*{Difficulties in higher dimensions}
When the dimension $n$ is greater than 3, there are several ways to attempt to generalize the Bing house, but usually we found ourselves leaving the world of PL immersions and instead building more singular maps. The main difficulty of the naive approach, which will become apparent Section \ref{sec: inductive argument}, revolves around a 2-disk of singular image points (i.e. points over which the map does not have full rank); a problem already present in dimension $n=4$. The map obtained from the naive approach might look, at places, like the projection $\mathbb{D}^2\times I \rightarrow \mathbb{D}^2$, which are problematic because it is unclear how to perturb away these singularities. 

Another attempt would be to look for these immersions in places in ``nature" where they might show up. So let us return to the geometry tale but this time we will look in higher dimensions.
It turns out that the $4$-sphere is also almost negatively curved (\cite{GGNP}). But again, seeing the 3-dimensional Bing house inside seems hopelessly difficult.

The goal of this paper is to give explicit constructions of higher dimensional Bing houses and PL immersions with contractible images of $(n-1)$-spheres for all $n\geq 3$ . 
\begin{theorem}\label{main theorem}
For every $n \geq 3$, there is a Bing house $Y^{n-1} \subset \mathbb{D}^n$.
\end{theorem}

\begin{remark}
Note that the theorem fails for $n = 2$, but again holds for $n=1$. 
\end{remark}

\subsection*{A cautionary example}
The suspension $\Sigma$ of the Bing house does not give an immersion of $\mathbb{S}^3$ since the retraction map has singularities (points not covered by an immersion) corresponding to the two cone points of $\Sigma$.
\newline

This paper is organized as follows. In Section \ref{sec: 3d Bing house} we give an explicit construction of a 3-dimensional Bing house. In section \ref{sec: inductive argument} we will describe two distinct routes to higher dimensional generalizations. One is to simply adapt the 3-dimensional construction to higher dimensions. The second is an inductive construction which yields a $Y^{n+2}$ from a $Y^n$. Given $X := Y^2$ and $Y^3$ as starting points, all dimensions $n \geq 2$ are thus accounted for.  

\begin{acknowledgement}
The second author is thankful to be at the Max Planck Institut f\" ur Mathematik in Bonn. She is grateful for its hospitality and financial support.
\end{acknowledgement}

\section{A $3$-dimensional Bing house}\label{sec: 3d Bing house}
In this section, we give an explicit construction of a 3-dimensional Bing house $Y^3$. Let us reserve $X$ for the classic 2-dimensional Bing house drawn in Figure 1. The deformation retraction of $\mathbb{D}^3$ to $X$ may be visualized as: ``insert two balloons through the tubes and inflate."
\newpage

Inverting about a point on the ``middle floor" the Bing house may alternatively be rendered in $\mathbb{S}^3$ as:
\begin{figure}[h!]
\label{2d-bing-inverted}
\centering
\includegraphics[scale=0.46]{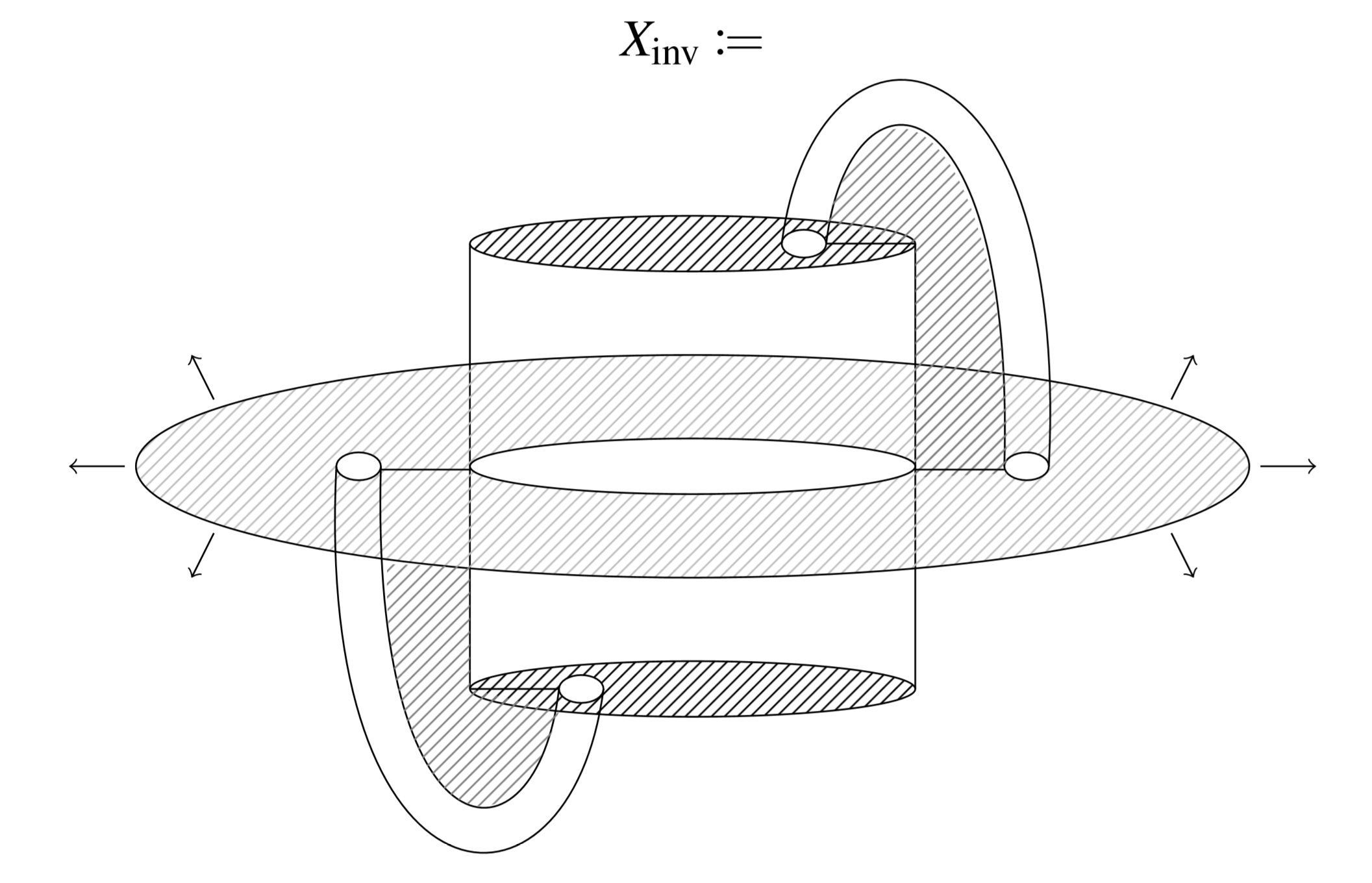}
\caption{}
\end{figure}

Our job now is to make precise a certain $3$-dimensional generalization of Figure \ref{2d-bing-inverted} which we sketch schematically in Figure 3.

\begin{figure}[h!]
\label{schematic}
\centering
\includegraphics[scale=0.55]{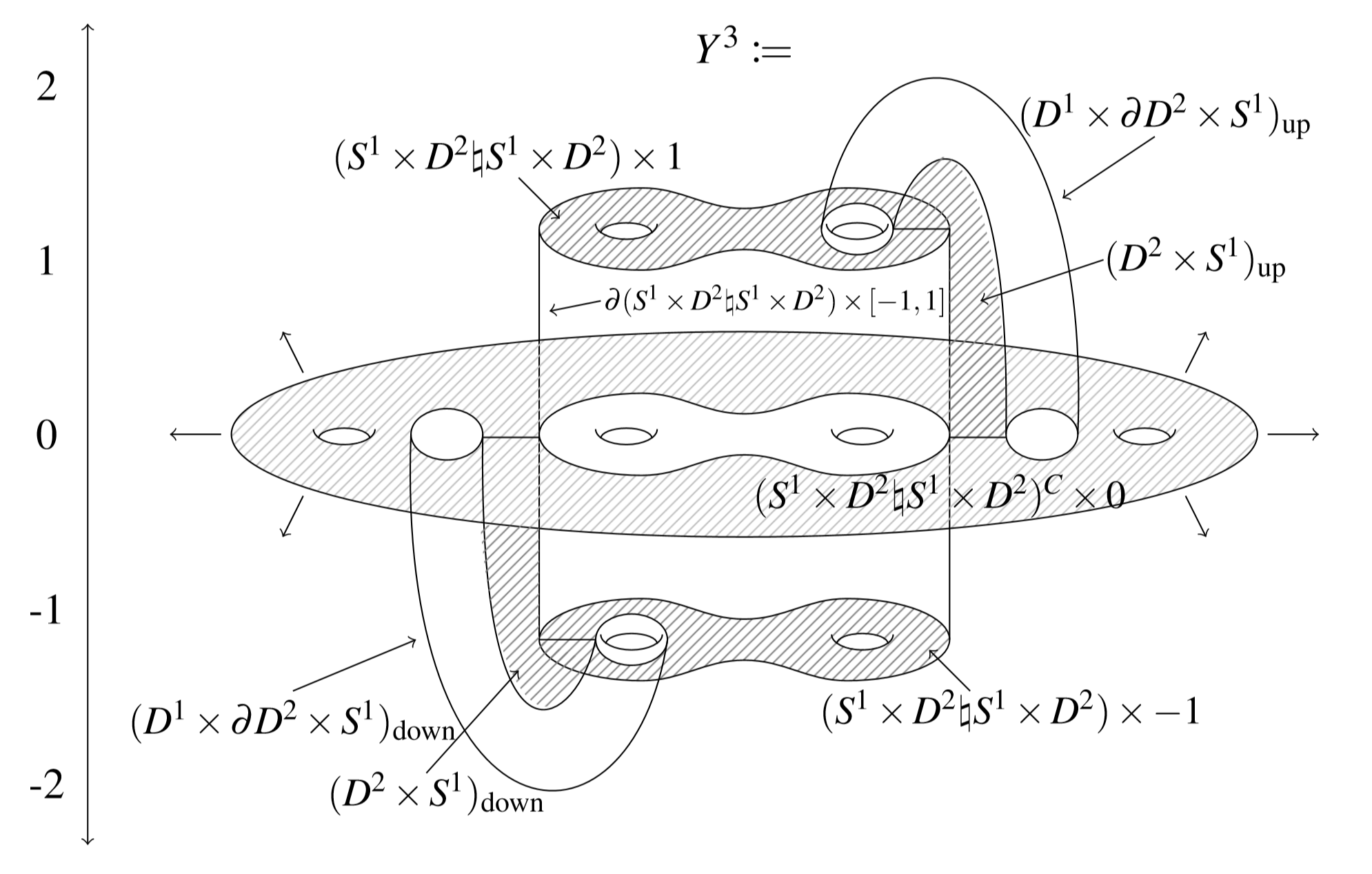}
\caption{}
\end{figure}

Here in text is a brief description of the $3$-dimensional ``pieces" of $Y^3$. The drawing in Figure 3 is in $\mathbb{S}^3 \times [-2, 2]$. Locate a genus 2 Heegaard surface $\Sigma$ in $\mathbb{S}^3 \times 0$. 
\newpage

\noindent
The first three pieces of $Y^3$ are: 
\begin{itemize}
\item[(1 $\&$ 2)] the \emph{interior} handle bodies to $\Sigma \times \{\pm 1\} \subset \mathbb{S}^3 \times \{\pm 1\}$ each with a (different) standard basis circle neighborhood deleted,
\item[(3)] and the exterior handle body to $\Sigma \times \{0\} \subset \mathbb{S}^3 \times \{0\}$ with two trivial (and unlinked) circle neighborhoods deleted. 
\item[(4 $\&$5)] The next two pieces are drawn as tubes (and we call them \emph{tubes}) but they are actually tubes cross circle $((\mathbb{D}^2 \times \partial \mathbb{D}^2) \times \mathbb{S}^1)_{\text{up}/\text{down}}$, where the extra circle factor parametrizes, longitudinally, the total of 4 deleted circles within the first three pieces.
\item[($6_{-}$ $\&$ $6_+$)] Next, add the vertical wall $\Sigma \times [-1, 1]$, which is actually two pieces $6_{-} = \Sigma\times[-1,0]$ and $6_{+} = \Sigma\times[0,1]$.
\item[(7 $\&$ 8)] Finally, the two pieces rendered as vertical disks are actually $\mathbb{D}^2 \times \mathbb{S}^1$'s with the extra circle again corresponding to the longitudinal parametrization of the 4 deleted circles.
\newline
\end{itemize}

So in total we have $3 + 2 + 2 + 2 = 9$ pieces, which we have labeled 1, 2, 3, 4, 5, $6_{-}$, $6_+$, 7 and 8. Each of these pieces, in its interior, will be covered $2\rightarrow 1$ by the map $r_1\colon \mathbb{S}^3 \rightarrow Y^3$. That is, implicitly we will be patterning $\mathbb{S}^3$ into 18 puzzle pieces which glue up (immersively) to form $Y^3$.
For reference, here are the pieces:

\begin{figure}[h!]
\label{}
\centering
\includegraphics[scale=0.51]{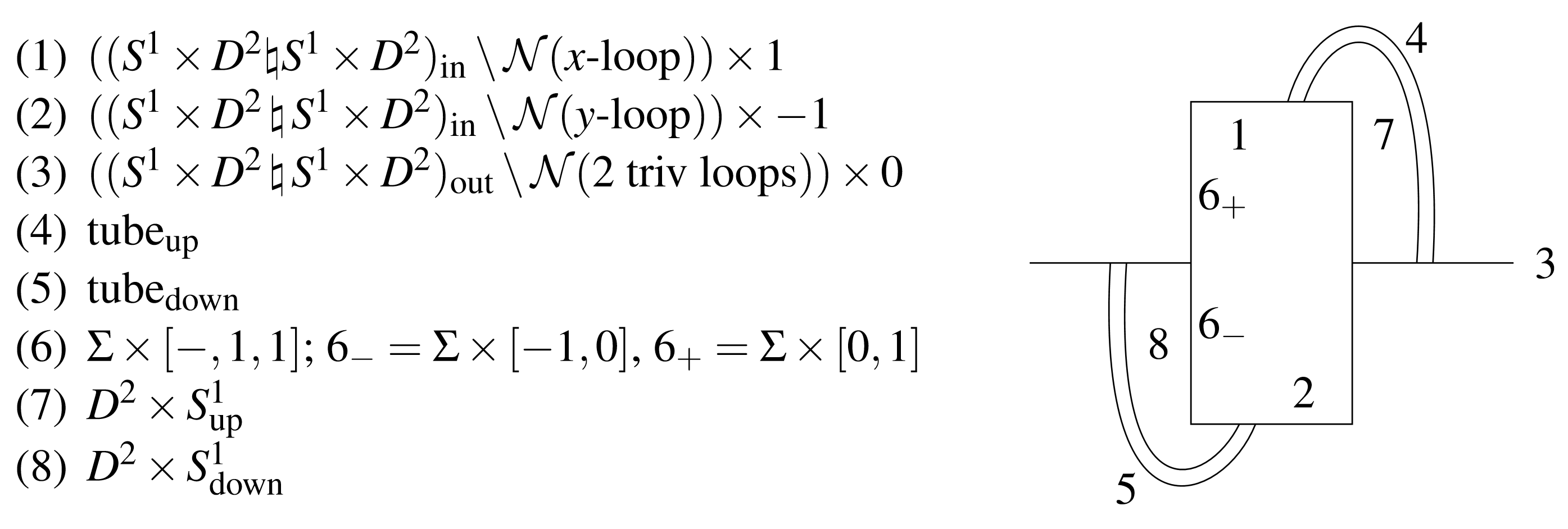}
\caption{}
\end{figure}

\bigskip 

Next, let us make it completely explicit where pieces 4, 5 and 7, 8 are located in $\mathbb{S}^3 \times [-2,2]$. By symmetry we only need to discuss pieces 4 and 7 (tube and $(\text{disk} \times \mathbb{S}^1)$), as pieces 5 and 8 are mirror images of these. Call $\tilde{S}^3$ the PL embedded copy of $\mathbb{S}^3$ consisting of 
\[((\text{inside } \Sigma) \times 1)\; \cup \;(\Sigma \times [0,1])\; \cup\; ((\text{outside } \Sigma) \times 0).\] 
In $\tilde{S}^3$ we may draw the boundary $\partial(\text{tube}_4)$ of $\text{tube}_4$:

\begin{figure}[h!]
\label{}
\centering
\includegraphics[scale=0.44]{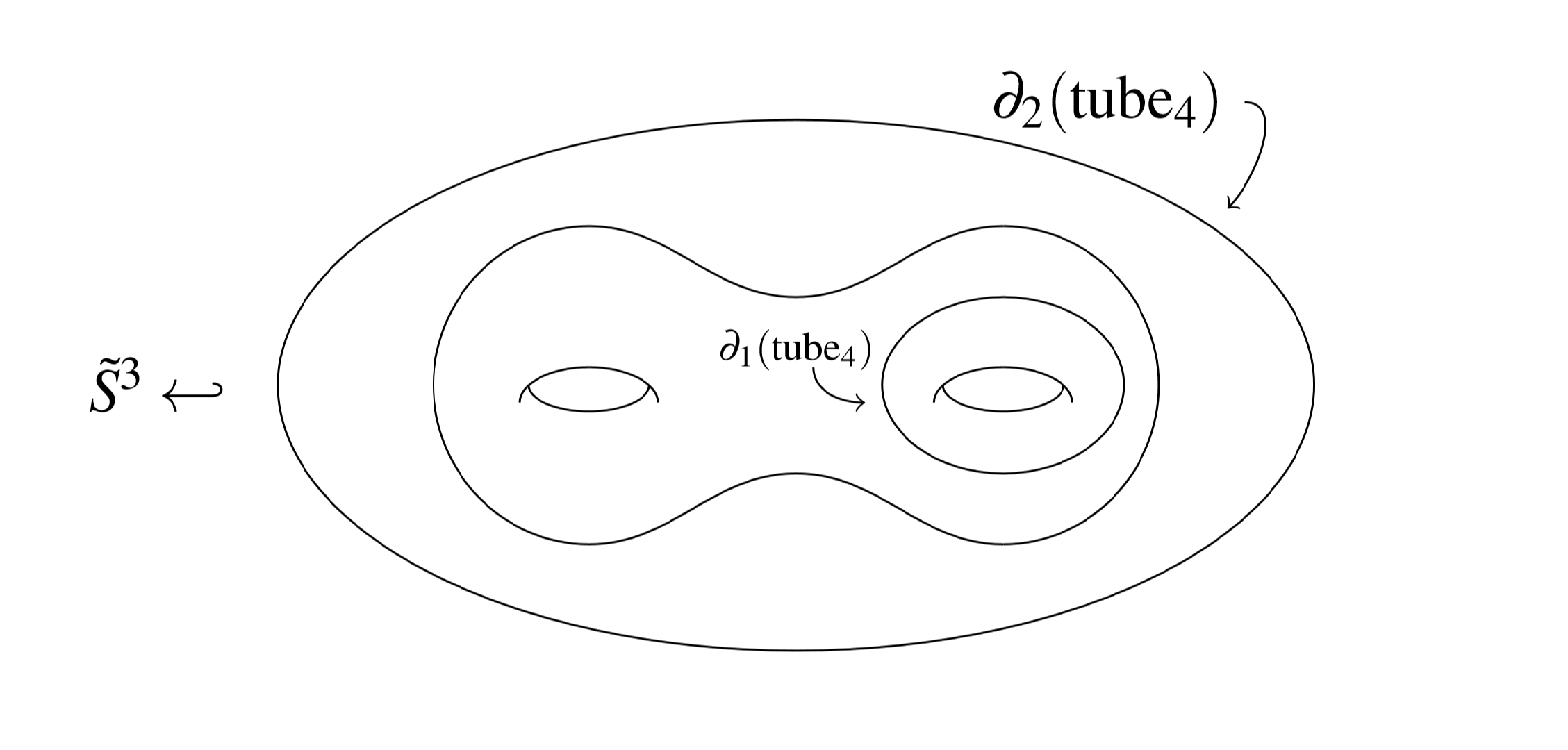}
\caption{}
\end{figure}

\[\;\]
\begin{figure}[h!]
\label{tube-2}
\centering
\includegraphics[scale=0.5]{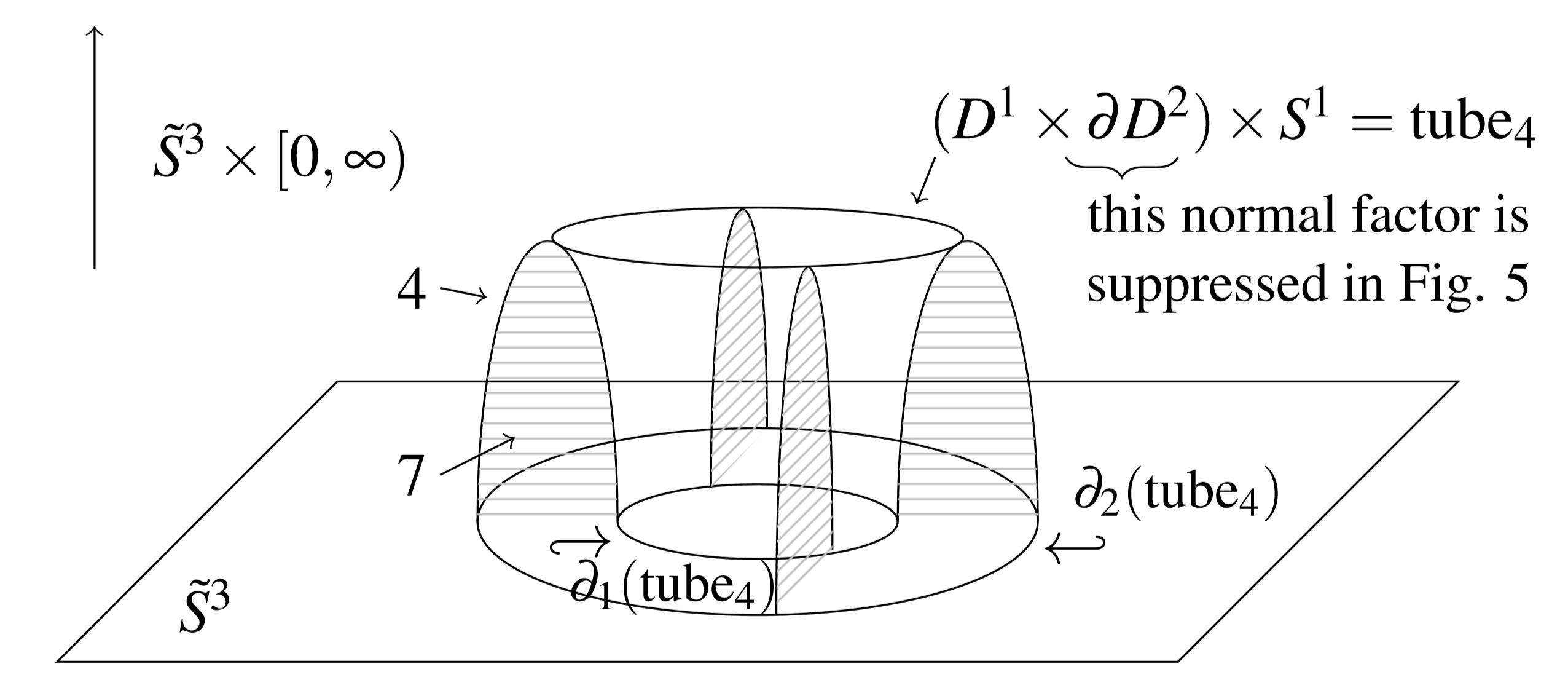}
\caption{The horizontal shades of the ``hoop" in Figure 6 are the $D^2$-factors in piece $7 \cong \mathbb{D}^2\times\mathbb{S}^1$.}
\end{figure}

\newpage

\begin{remark}
A quick check verifies that $Y^3$ is simply connected and has vanishing integral homology and hence is contractible. 
\end{remark}

Our next step is to build a (generically) $2 \rightarrow 1$ immersion of some closed 3-manifold $M$ onto $Y^3$. We will construct the source 3-manifold $M$ piece by piece and at the end verify it is indeed a 3-sphere. It is not much extra work simultaneously describe the regular neighborhood $\mathcal{N}$ of $Y^3 \subset \mathbb{S}^3 \times [-2,2]$. What we will be verifying is that $\mathcal{N} \overset{\text{PL}}{\cong} \mathbb{D}^4$ and $\partial \mathcal{N} \overset{\text{PL}}{\cong} \mathbb{S}^3$, and that for a natural mapping cylinder structure $f\colon \partial \mathcal{N} \rightarrow Y^3$ on $\mathcal{N}$, $f$ is a PL immersion.
\newpage

\subsection*{Assembling the pieces into a 3-sphere}

Kirby calculus notation for $1 \cup 2 \cup 6$ is:
\begin{figure}[h!]
\label{}
\centering
\includegraphics[scale=0.37]{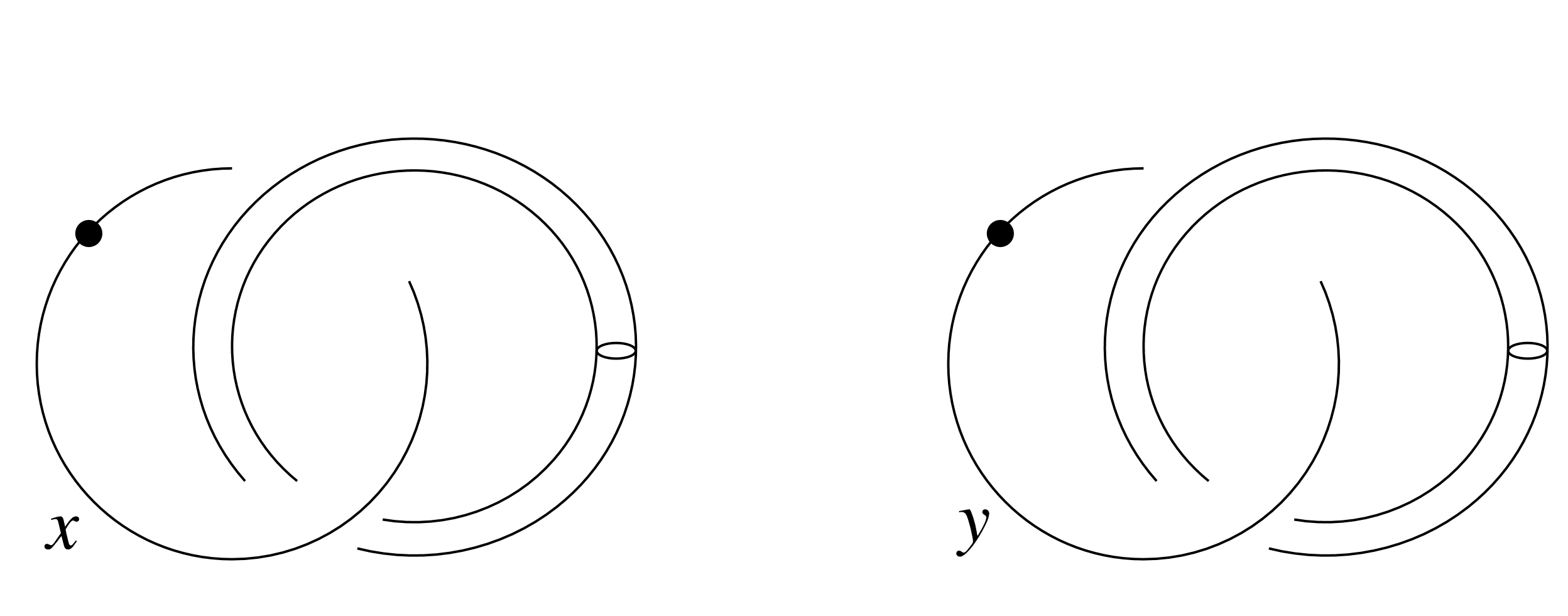}
\caption{}
\end{figure}

\bigskip

Now let us understand $4 \cup 3 \cup 6_- \cup 2 \cup 5 \cup 8_2$ (and its symmetric partner $5 \cup 3 \cup 6_+ \cup 1 \cup 4 \cup 7_2$). The subscripts mean use two copies of 8 and 7. Notice with this convention that every piece is now used twice corresponding to the (generic) $2 \rightarrow 1$ nature of $r_1$.

Start with $4 \cup 3 \cup 6_- \cup 2$. The $\text{tube}_4$ merely adds a collar, and $3 \cup 6_- \cup 2$ is simply a genus 2 Heegaard decomposition of a $3$-sphere $\mathbb{S}^3$ with the neighborhood of two trivial circles deleted from the exterior handlebody. At the level of Kirby diagrams we have:

\begin{figure}[h!]
\label{}
\centering
\includegraphics[scale=0.51]{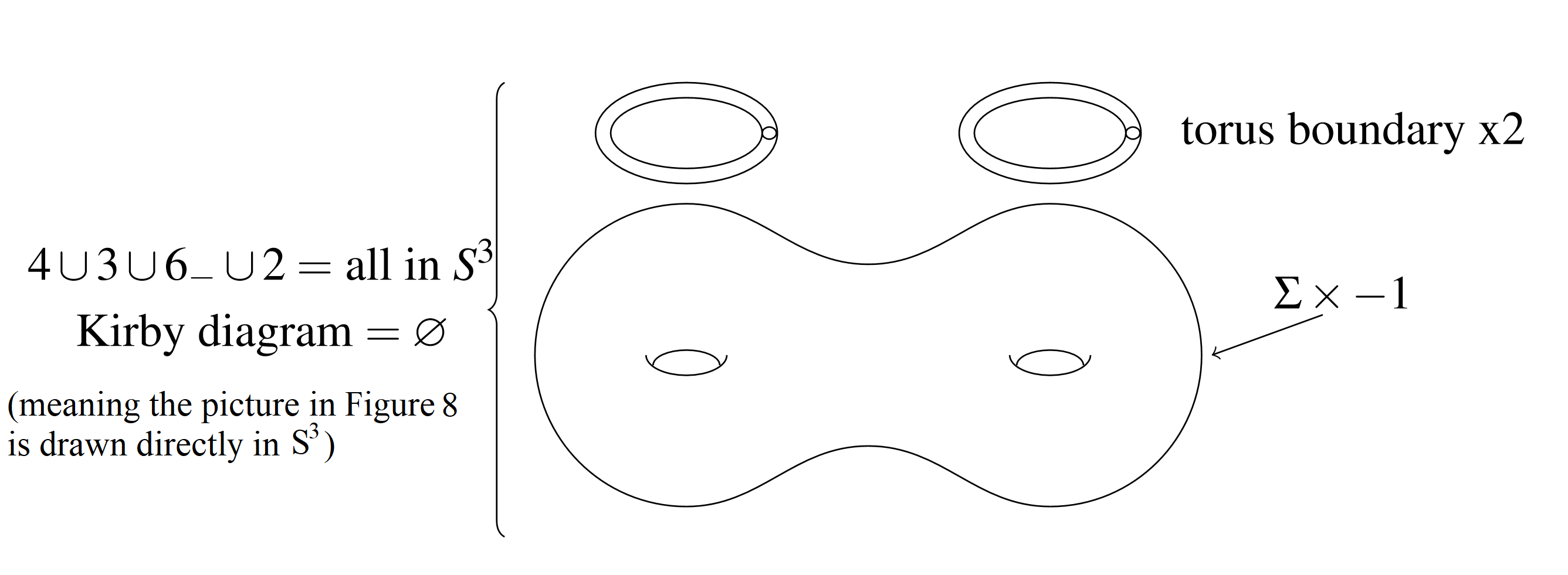}
\caption{}
\end{figure}

\bigskip

The $\text{tube}_5$ effects the attachment of a trivial ``round 1-handle" to $\mathbb{S}^3$ (a round 1-handle in this dimension is $(\mathbb{D}^1 \times \mathbb{D}^2, \partial \mathbb{D}^1 \times \mathbb{D}^2) \times \mathbb{S}^1$) and the two copies of 8 effects a round 2-handle attachment ($(\mathbb{D}^2 \times \mathbb{D}^1, \partial \mathbb{D}^2 \times \mathbb{D}^1) \times \mathbb{S}^1$). Looking back at Figure 6, observe that the regions labeled 5 and 8 constitute a (Morse canceling pair)$\times \mathbb{S}^1$. That is a circle's worth of canceling critical points of index 1 and 2 (detail in Figures 8 and 9 below). Now adding 5 yields:
\begin{figure}[h!]
\label{}
\centering
\includegraphics[scale=0.4]{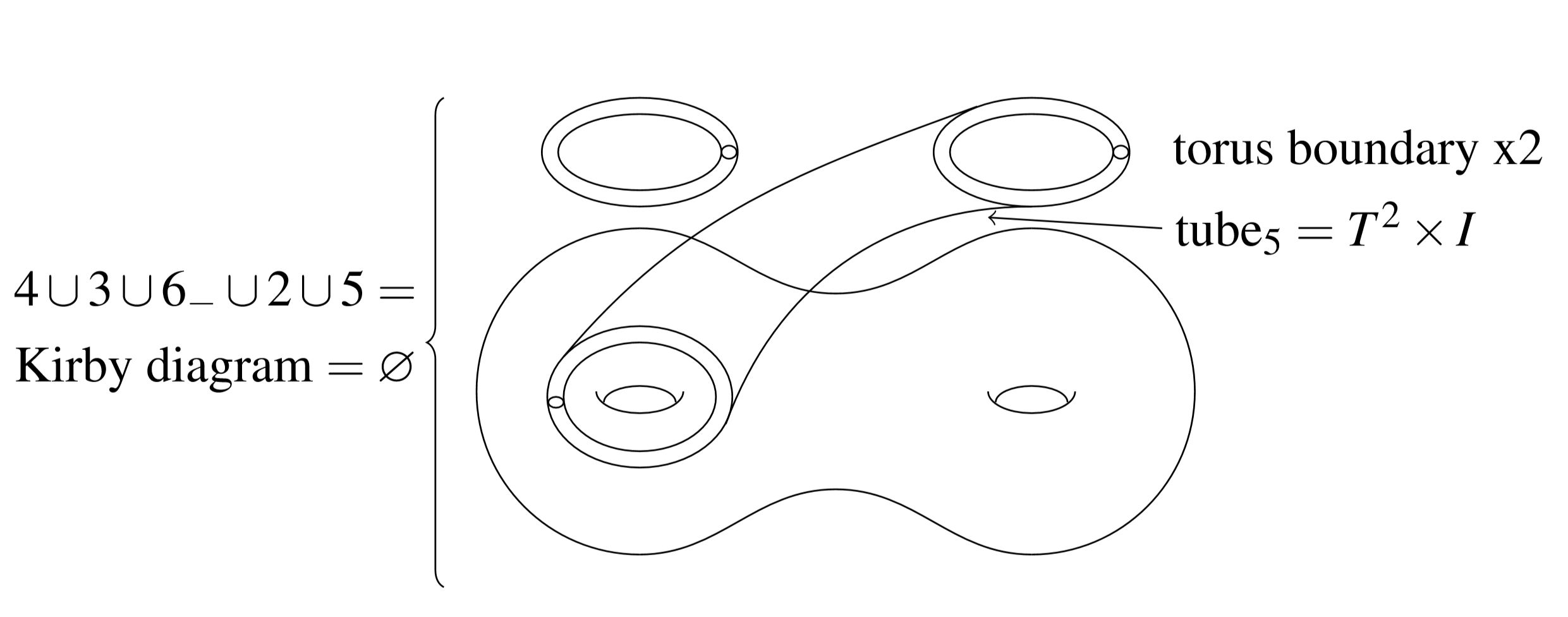}
\caption{}
\end{figure}

\begin{figure}[h!]
\label{}
\centering
\includegraphics[scale=0.37]{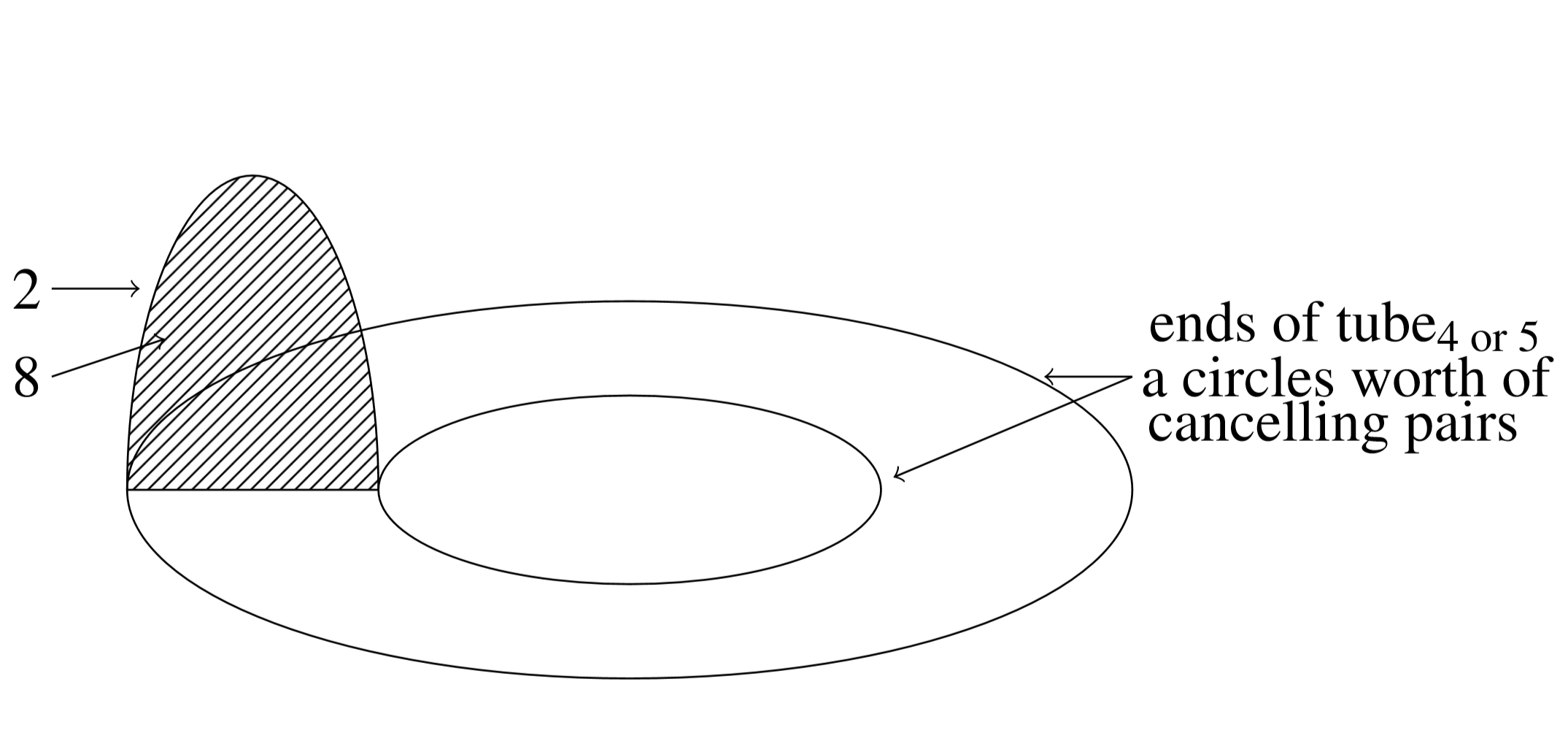}
\caption{}
\end{figure}

\newpage 

The upshot is that both $4 \cup 3 \cup 6_- \cup 2 \cup 5 \cup 8_2$ and its symmetric partner $5 \cup 3 \cup 6_+ \cup 1 \cup 4 \cup 7_2$ are 3-spheres minus the neighborhood of an unknotted simple closed curve (i.e. solid tori). Furthermore, the meridians of these two solid tori glue to the $x$ and $y$ longitudes of $1 \cup 2 \cup 6$, surgering the $3$-manifold represented by the central rectangle in Figure 4,
\[ \partial ((\text{genus 2 handlebody})\times I)\cong (\mathbb{S}^1 \times \mathbb{S}^2) \# (\mathbb{S}^1 \times \mathbb{S}^2)\] 
to $\mathbb{S}^3$. This completes the proof that the 18 pieces in fact glue up to $\mathbb{S}^3$.

\subsection*{PL regular neighborhood of $Y^3$} We may now proceed to draw 4-dimensional Kirby diagrams for the PL regular neighborhood of the image $Y^3$.

\[\mathcal{N}(1 \cup 2 \cup 6) = (\mathbb{S}^2 \times \mathbb{D}^2) \natural (\mathbb{S}^1 \times \mathbb{D}^3) \natural (\mathbb{S}^1 \times \mathbb{D}^3),\] or in Kirby calculus:
\begin{figure}[h!]
\label{}
\centering
\includegraphics[scale=0.27]{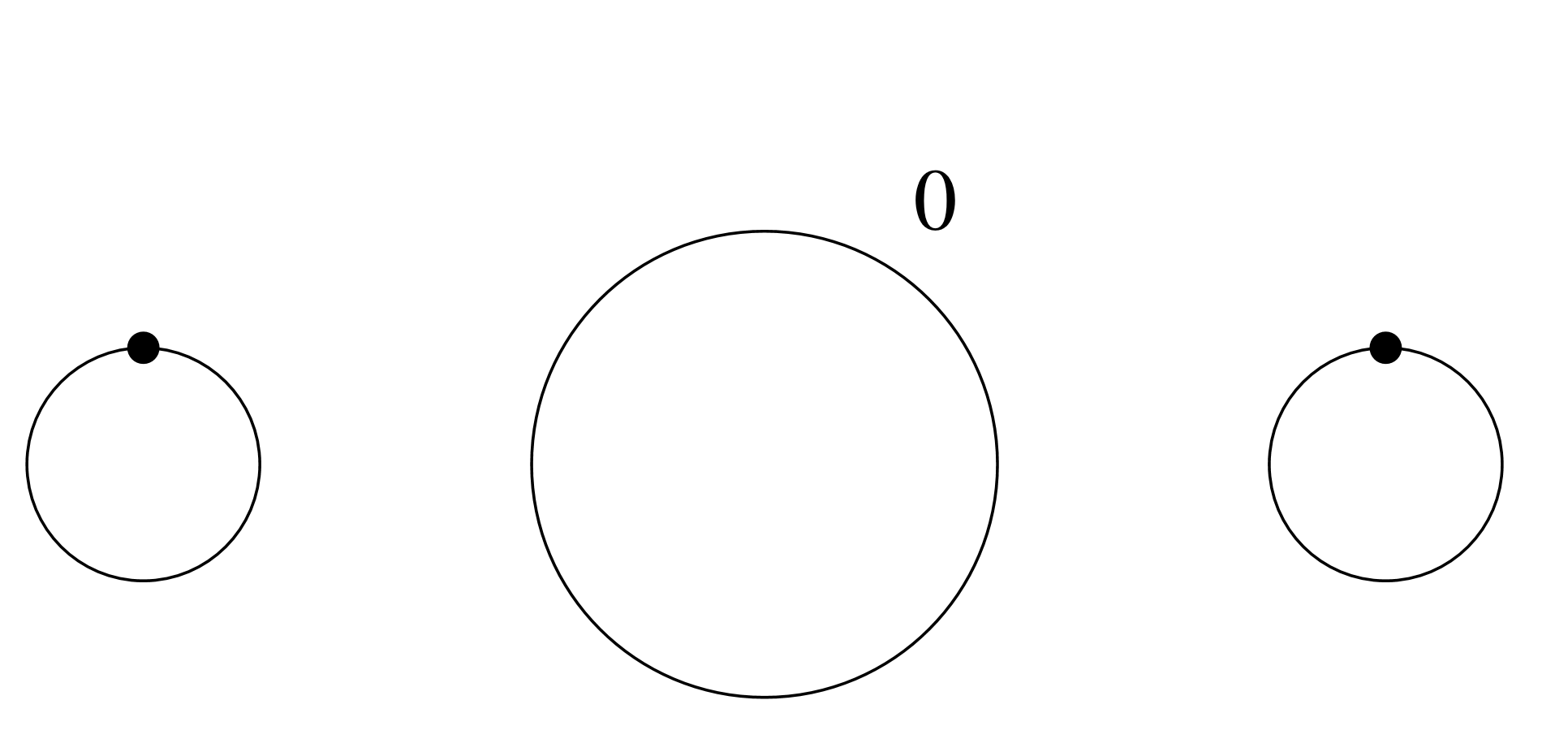}
\caption{}
\end{figure}
\newpage

We now proceed to add the remaining thickened pieces and draw the corresponding Kirby diagrams---omitting the 3-handles.
Due to a theorem of Trace \cite{Trace} based on work of Laudenbach and Poenaru \cite{LP} (see remark 4.4.1 of \cite{Gompf} for an exposition) it is not necessary to record the 3-handle attaching 2-spheres of a 4-dimensional handlebody with connected boundary in order to define the diffeomorphism type of the 4-manifold. Thus tracking only 1 and 2-handles, we get Figure 12.

\begin{figure}[h!]
\label{}
\centering
\includegraphics[scale=1.01]{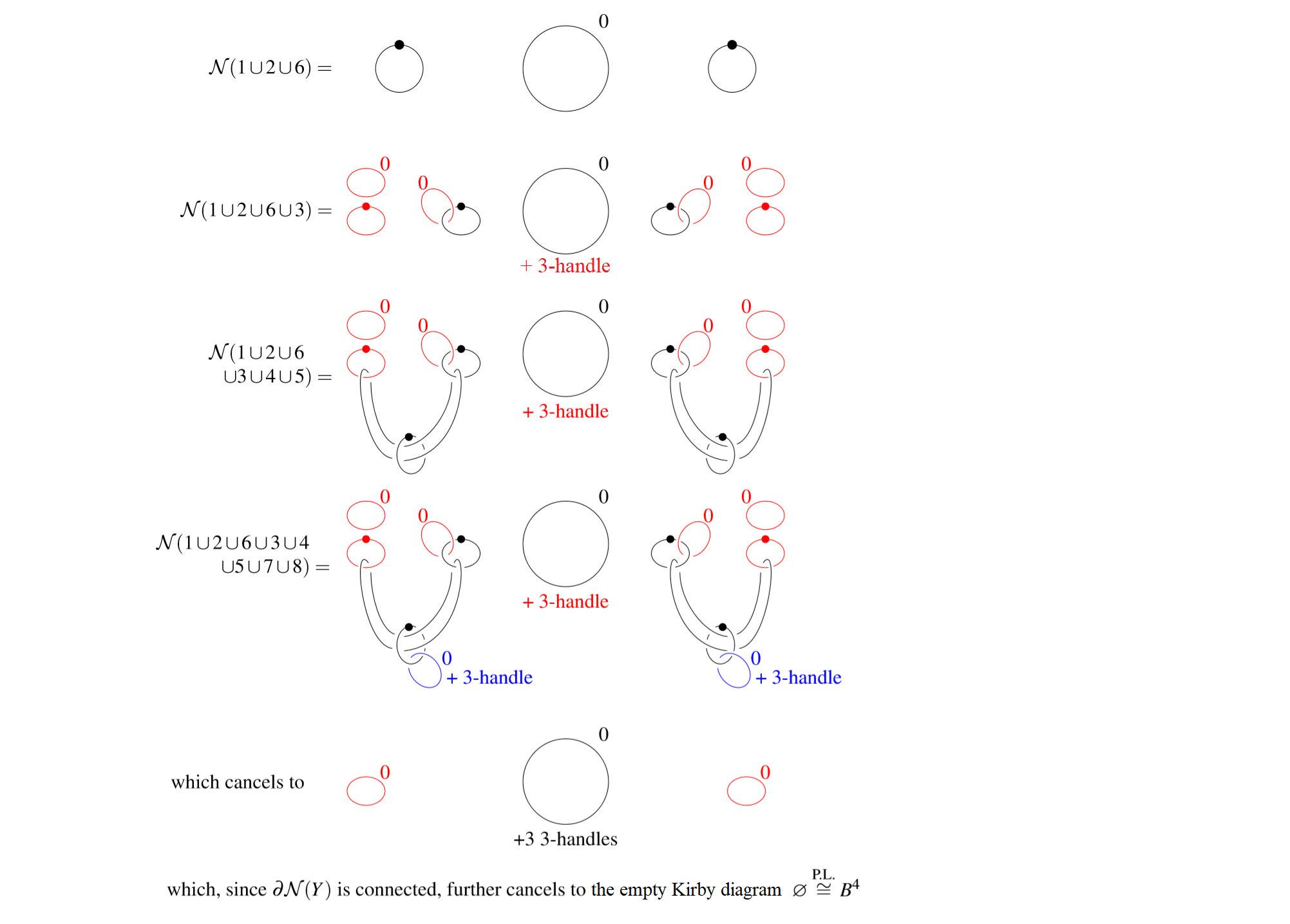}
\caption{}
\end{figure}

This completes the construction of a 3-dimensional Bing house $Y^3 \subset \mathbb{D}^4$. 

The conclusion that the natural PL mapping cylinder structure on $\mathcal{N}(Y^3)$ defines a PL-immersion from $\mathbb{S}^3 = \partial \mathcal{N}(Y^3) \rightarrow Y^3$ follows from inspection. There are three local (PL) models for a point $y \in Y^3$. It can be 
\begin{itemize}
\item[(1)] in the interior of a sheet,
\item[(2)] at a $(\text{triple point } Y) \times \R^2$, or
\item[(3)] (the cone on the 1-skeleton of tetrahedron)$\times \mathbb{R}$.
\end{itemize}
Unlike at the cone points of the suspension of the Bing house, for each of these models a neighborhood of $y$ is covered by a finite number of $\mathbb{R}^{n-1}$-sheets (sheets of the immersion of $\mathbb{S}^{n-1}$). 
The approach via mapping cylinder lines to any such $y$ from any complementary region is locally (PL) one-to-one in a neighborhood of $y$. That is, $Y^3$ is covered by PL immersed $\R^3$-charts. This completes the proof of Theorem \ref{main theorem} for $n=3$ and $4$.

\section{To higher dimensions}\label{sec: inductive argument}

Now we come to a fork in the road, and we will take both paths. One is by relabeling the pieces in Figure 3. The other one is an inductive construction.

\subsection{Relabeling Figure 3}
To produce Bing houses $Y^{n-1}$ in $\mathbb{D}^n$ for $n \geq 5$, we can explain a straightforward relabeling of Figure 3 appropriate for the higher dimension. This is Figure 13. 
\begin{figure}[h!]
\label{}
\centering
\includegraphics[scale=0.5]{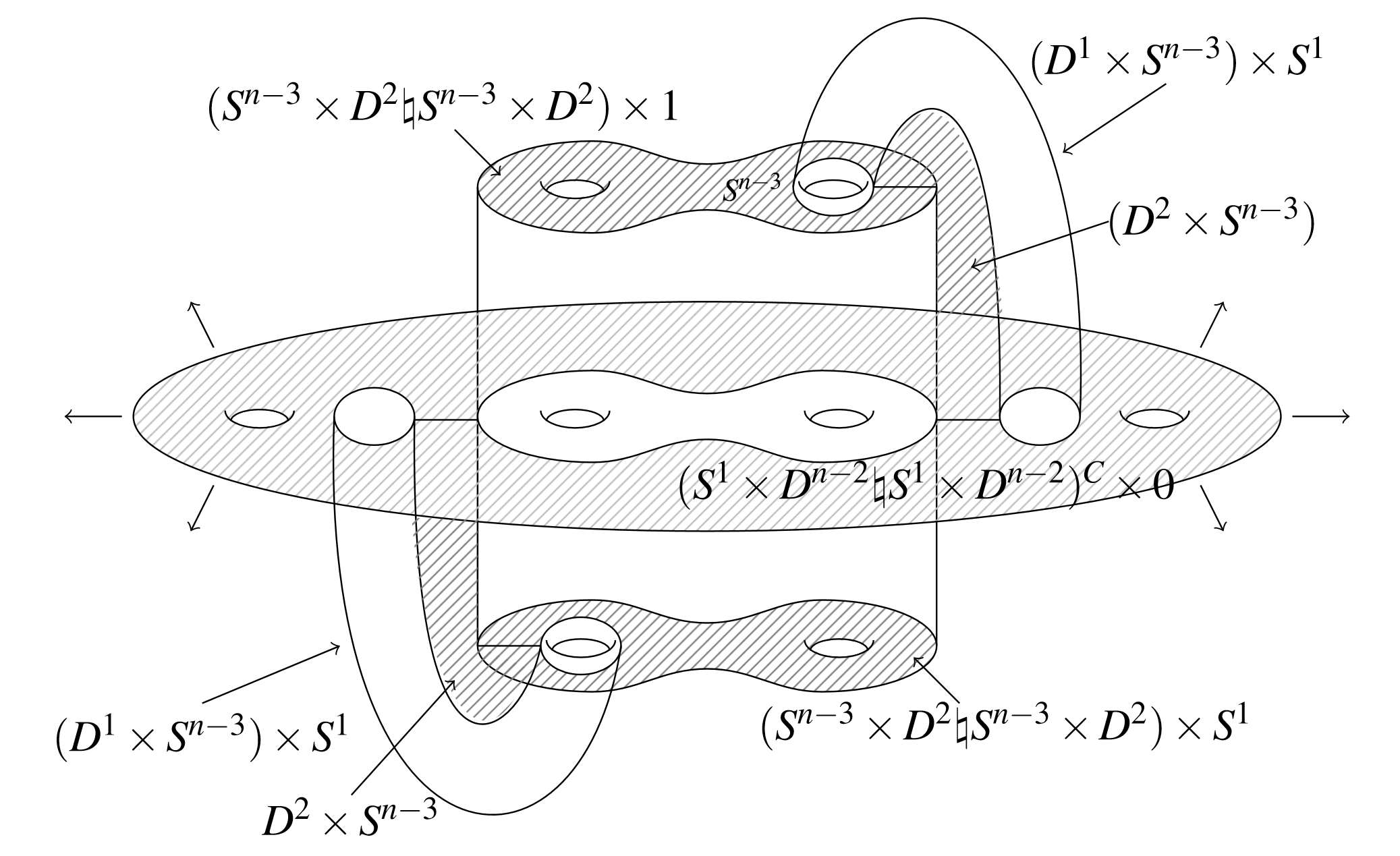}
\caption{}
\end{figure}
The proof is parallel to what we did for $n=4$ (and will not be repeated) up to the Kirby calculus discussion summarized by Figure 12. Doing the handlebody theory carefully in higher dimension would require establishing the notations for a generalized Kirby calculus and drawing \emph{all} handles since we no longer have the convenience of Trace's Theorem. Fortunately all this is unnecessary, since the PL Poincare conjecture can be used to conclude that for $n \geq 5$, $\mathcal{N}(Y^{n-1}) \overset{\text{PL}}{\cong} \mathbb{D}^n$. This is because the lower dimensional argument extends far enough to imply $Y^{n-1} \cong \{\ast\}$, and $\partial \mathcal{N}(Y^{n-1}) \overset{\text{PL}}{\cong} \mathbb{S}^{n-1}$. So in pursuing this first fork we are content to leave the reader with Figure 13.

\subsection{An inductive construction}
The second branch of our road to high dimensions is to give an inductive procedure for producing a Bing house $Y^{n+2}$ given a Bing house $Y^n$ as starting point. Because the induction ``skips over" a dimension we need both the classical $X = Y^2$ and the just constructed $Y^3$ as starting points. We explain the induction in some detail.

Let us define a \emph{rank-2 Bing hut} as something a bit weaker than a Bing house. It will be PL complex $Z^{n-1} \subset \mathbb{D}^n$ with a PL regular neighborhood, $\mathcal{N}(Z)$, PL homeomorphic to a PL $n$-ball $\mathbb{D}^n$. It is convenient to let $Z^{n-1} \cap \partial \mathbb{D}^n = \mathbb{S}^{n-2}$, the equator between hemispheres $N$ and $S$ of $\partial \mathbb{D}^n = \mathbb{S}^{n-1}$. Furthermore, $\mathcal{N}(Z)$ must possess a mapping cylinder structure induced by a map $g: \partial \mathcal{N}(Z) \cong \mathbb{S}^{n-1} \rightarrow Z$ which is required to be a PL immersion on the interior $\dot{A}$ of a compact co-dimension zero PL submanifold $A$ of $\partial \mathcal{N}(Z)$. The closure $\overline{\partial \mathcal{N} \setminus A}$ is further required to consist of finitely many PL copies of $\mathbb{D}^2 \times \mathbb{S}^{n-3}$, and $g$ restricted to each copy, is the projection (modulo the selection of suitable PL coordinates in domain and range) $D^2 \times \mathbb{S}^{n-3} \xrightarrow{\pi_1} D^2$, $D^2 \subset Z$ (see Figure 14). In other words $\mathcal{N}(Z)$ contains distinguished 2-handles and on these the mapping cylinder structure is the rank 2 projection.

\begin{figure}[h!]
\label{local-model}
\centering
\includegraphics[scale=0.42]{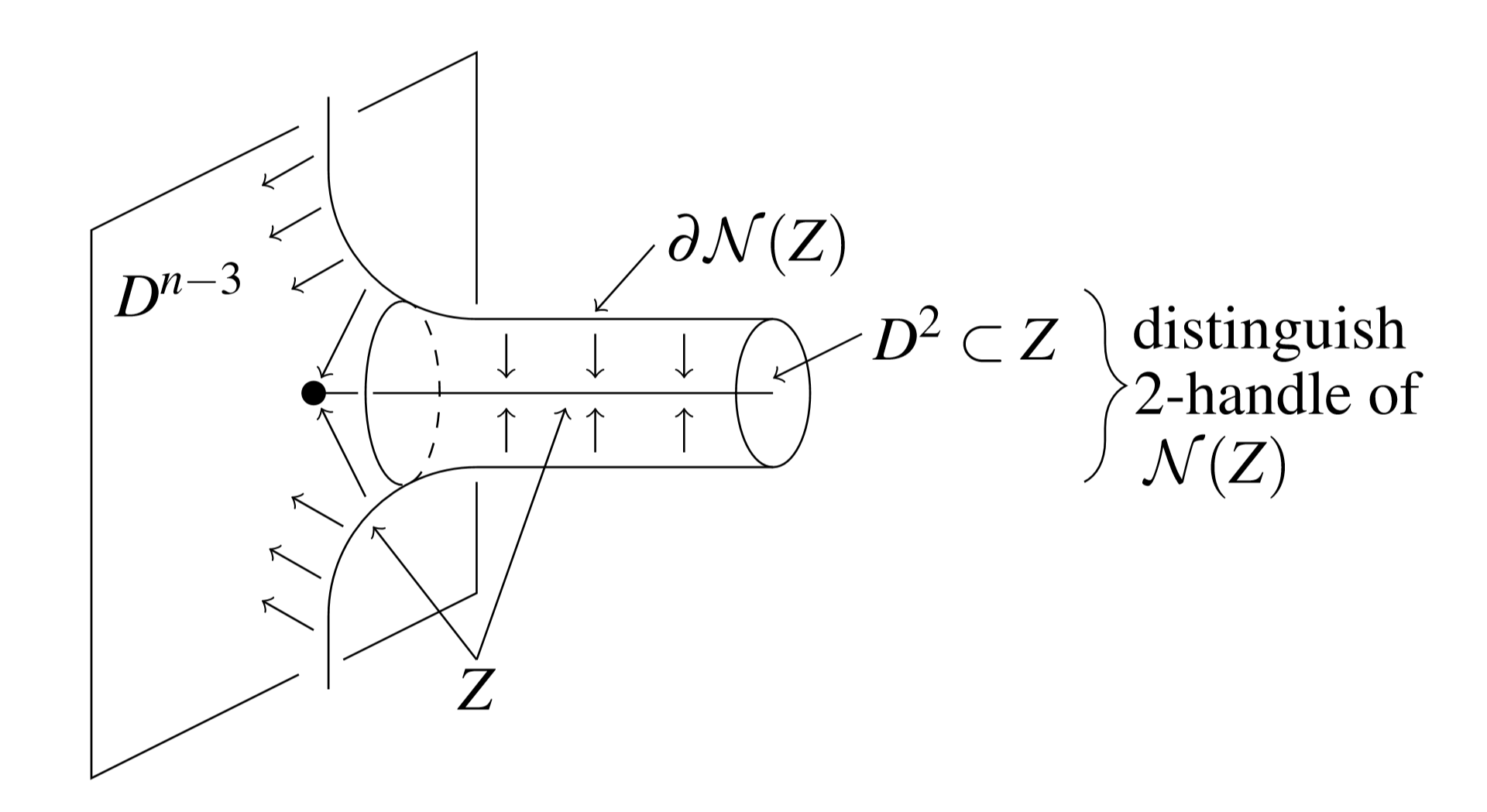}
\caption{Local model for $g$ near 2-handles $(\mathbb{D}^2\times \mathbb{D}^{n-2}, \partial \mathbb{D}^2\times \mathbb{D}^{n-2})$.}
\end{figure}

The (second) proof of Theorem \ref{main theorem} will be complete by combining  Lemmas \ref{rank-2-bh} and \ref{making-bh} below.

\begin{lemma}\label{rank-2-bh}
For all $n \geq 3$ there is a rank-2 Bing hut $Z^{n-1} \subset \mathbb{D}^n$.
\end{lemma}

\begin{lemma}\label{making-bh}
If $Y^{n-3} \subset \mathbb{D}^{n-2}$ is a $(n-3)$-dimensional Bing house and $Z^{n-1} \subset \mathbb{D}^n$ an $(n-1)$-dimensional rank-2 Bing hut, then replacing the distinguished disks $D^2_s \subset Z^n$ with copies of  $\mathbb{D}^2 \times Y^{n-3}$, we obtain an $(n-1)$-dimensional Bing house $Y^{n-1} = (Z^{n-1} \setminus D_s^2) \cup (D_s^2 \times Y^{n-3})$. One level of pairs $(\mathcal{N}(Y^{n-1}), Y^{n-1}) = (\mathcal{N}(Z^{n-1}) \setminus \text{2-handles}, (Z^{n-1} \setminus D_s^2) \cup D_s^2 \times Y^{n-3})$.

The subscript ``$s$" labels the 2-handles over which the original projection is rank-two, in our case there are precisely two such 2-handles so $s\in \{1,2\}$.
\end{lemma}

\begin{proof}[Proof sketch of Lemma \ref{rank-2-bh}]
	$Z^{n-1}$ is built in a few steps in Figure 15.
	
\begin{figure}[h!]
\label{building-bh}
\centering
\includegraphics[scale=0.75]{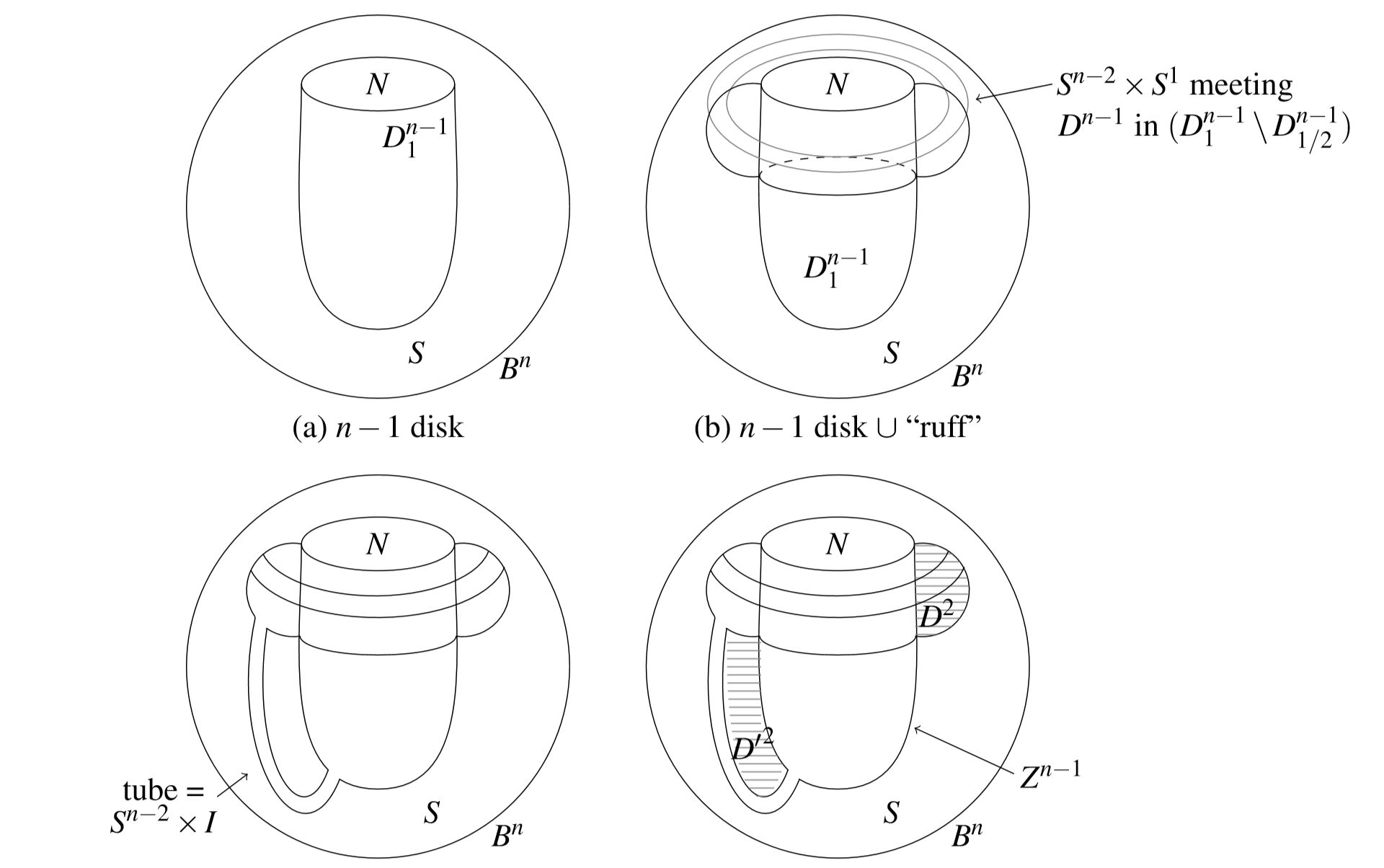}
\caption{}
\end{figure}

One may now check that Figure 15(d), in fact, is a rank-2 Bing hut. To aid in this checking, note two maps: the fold map $f_1: \mathbb{S}^{n-1} \ra \mathbb{D}^{n-1}$ takes the hemispheres $N, S \subset \mathbb{S}^{n-1}$ each 1-1 onto $\mathbb{D}^{n-1}$, and the natural degree one map $f_2: \mathbb{S}^{n-1} \rightarrow (\mathbb{S}^{n-2} \times \mathbb{S}^1) \cup D^2$ where $D^2$ attaches to $\ast \times \mathbb{S}^1$ degree one.

The north polar region $N$ of $\mathbb{S}^{n-1} = \partial B^n$ maps to the ``inside" of $Z^{n-1}$ essentially by $f_1\vert_N \#_{\text{tube}} f_2$. The map to the ``outside" of $Z^{n+1}$ is built by rounding $f_1\vert_S$ along the ``ruff" $S^{n-2} \times S^1$. However, the hemisphere $S$ will not retract to the outside of $Z^{n-1}$ until the element of fundamental group generated by the tube is canceled by the disk labeled $D'^{2}$. The 2-handle of $\mathcal{N}(Z^{n-1})$ with $D'^{2}$ as core surgers the outer boundary from $(\mathbb{S}^1 \times \mathbb{S}^{n-1} \setminus \mathbb{D}^{n-1})$ to $(\mathbb{S}^{n-1} \setminus \mathbb{D}^{n-1}) \cong \mathbb{D}^{n-1}$.

	Using PL coordinates to construct the two distinguished $n$-dim 2-handles in $(\mathcal{N}(Z^{n-1}), Z^{n-1})$ we see that the map giving mapping cylinder coordinates to $(\mathcal{N}(Z^{n-1}),Z^{n-1})$ has the form shown in Figure 14. In particular, restricted to the ``belt region" of the 2-handle, $\mathbb{D}^2 \times \mathbb{S}^{n-3}$, $g\colon \mathbb{S}^{n-1}\rightarrow Z$ assumes the form of first coordinate projection, ``bent" toward the attaching region as shown in Figure 16 (with $\mathbb{D}^2$ dimensionally reduced to $\mathbb{D}^1$). 
\begin{figure}[h!]
\label{lemma-figure}
\centering
\includegraphics[scale=0.35]{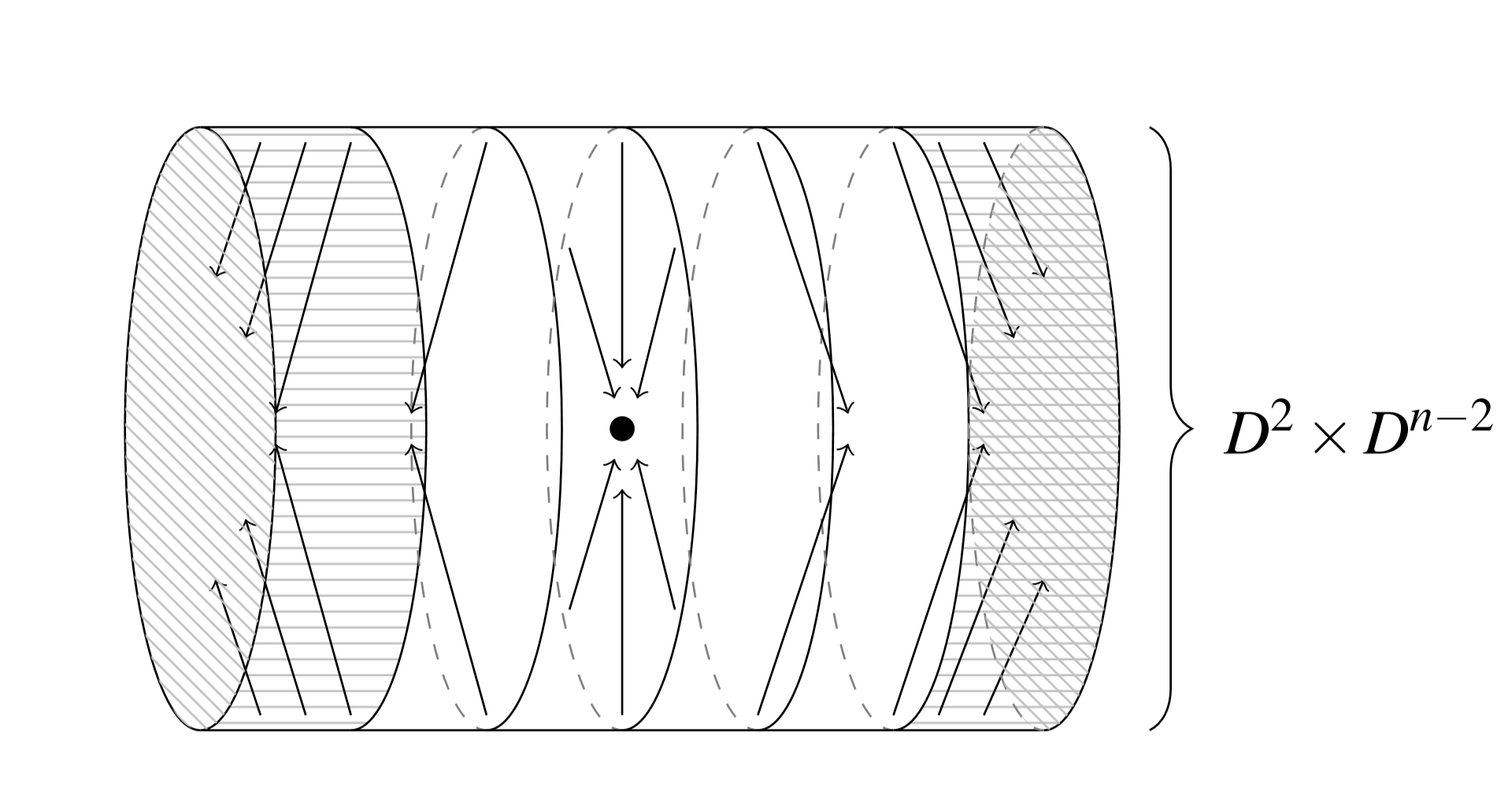}
\caption{Arrows represent the rank-2 map $g$.}
\end{figure}
The shaded boundary collar of the belt parametrizes (under $g$) the attaching region of the distinguished 2-handle.

This completes our sketch of Lemma 1.
\end{proof}

\begin{proof}[Proof sketch of Lemma \ref{making-bh}]
We use the letter $Y$ now also a glyph for the Bing house $Y^{n-1}$. We may now insert into the mapping cylinder structure shown in Figure 16 a copy of the $(n-2)$-dimensional mapping cylinder structure on $(\mathcal{N}(Y^{n-1}), Y^{n-1})$. This is done fiber-wise over each preimage of $p \in D^2$ under the ``bent" $g$ illustrated in Figure 16. The result is:

\begin{figure}[h!]
\label{lemma-2-figure}
\centering
\includegraphics[scale=0.4]{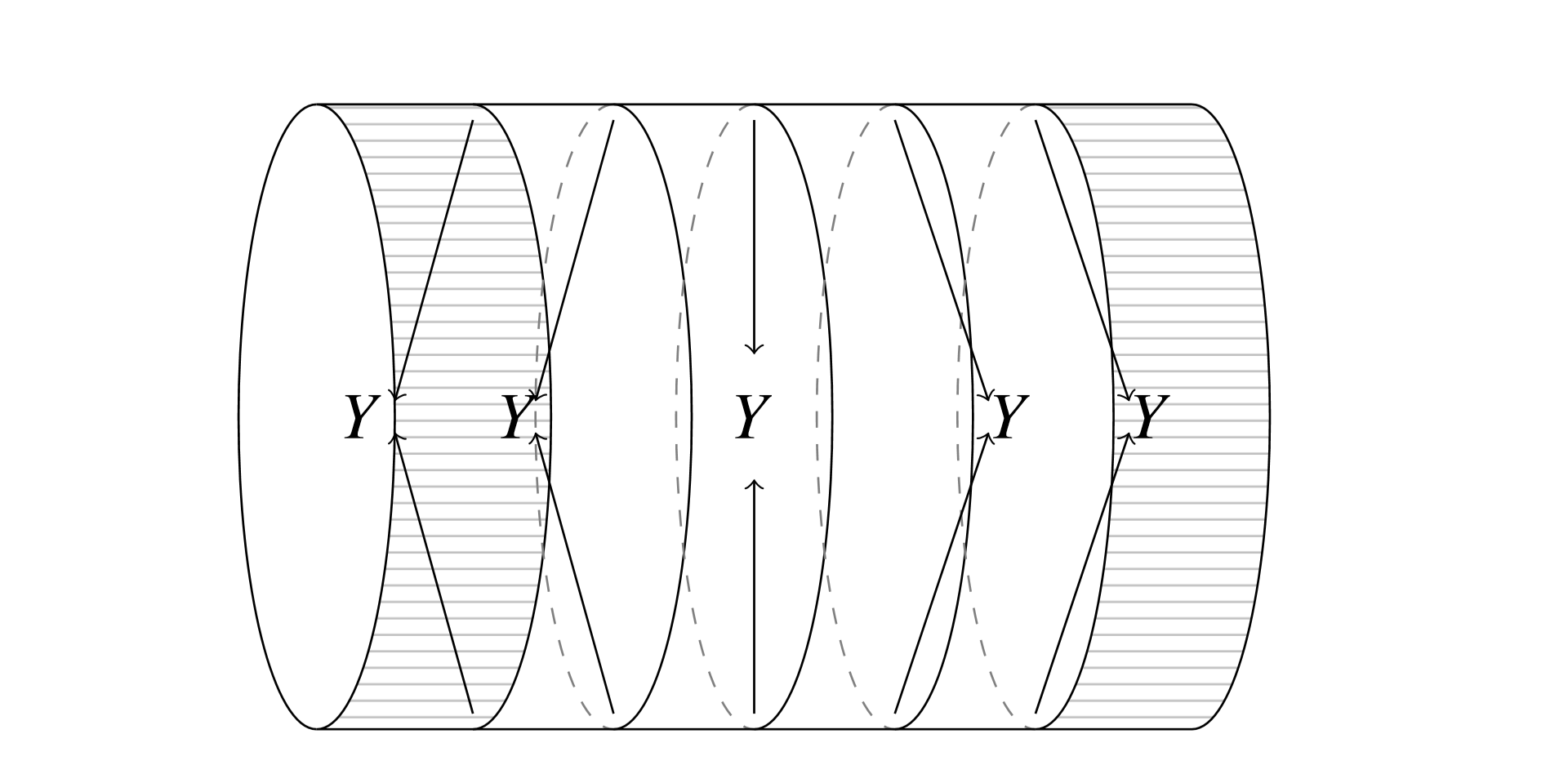}
\caption{}
\end{figure}

This substitution restores the immersive property and the rank-2 hut becomes a house. In this argument it is not necessary to quote the high dimensional PL Poincare conjecture since we have constructed the collapse of $\mathbb{D}^n$ to $Y^{n-1}$ explicitly.
\end{proof}

\section{Appendix}	
	
\begin{theoremA}
A $C^1$-immersion $f \colon \mathbb{S}^2 \looparrowright \mathbb{D}^3$ cannot parametrize a Bing house $Y$.
\end{theoremA}

\begin{proof}
The map $f$ determines the following equivalence relation $T$ on $\mathbb{S}^2$, which we call \emph{tangent-equivalent}. Two points $a, b \in \mathbb{S}^2$ are \emph{tangent-equivalent} if and only if they have identical $1$-jets under $f$. That is, $f(a) = f(b)$ and $df\mid_a = df\mid_b$. Clearly, this is an equivalence relation. 

Let $X =\mathbb{S}^2/T$ be the quotient space. Let $Y$ denote the image of $f$. Since we have no genericity assumptions on $f$, both $X$ and $Y$ may be rather ``wild" topological spaces. To see this, note that for any closed set $C$ we may find a smooth function vanishing to first order precisely on $C$.  An intermediate example between ``wild" and ``reasonable" is the union of the $xy$-plane and the surface obtained by revolving the graph of $z = x^4(\sin(1/x) -1)$ around the $z$-axis. Locally two sheets of $Y$ could meet tangentially at the closed subset consisting of concentric circles of radii $1/k$ for $k\in \mathbb{Z}^+$ and the center of these circles. When $f$ is ``reasonable", one should think of $X$ as a finite $2$-complex.

However, despite how ``wild" $X$ is, there is always  an obvious notion of a tangent space at every point $x\in X$, even when $x$ is not a manifold-point. Thus, there is an obvious notion of a \emph{$C^1$ (or smooth) immersion} $X\looparrowright \mathbb{R}^3$. 

It is also clear that $f$ factors as a composition $f = h\circ g $ of $C^1$ immersions:
$$ \mathbb{S}^2  \stackrel{g}\looparrowright X \stackrel{h}\looparrowright Y.$$

One can think of the space $X$  as a ``wild branched surface"\footnote {It is what is called a \emph{non-singular branched $n$-manifold} in \cite{Williams}.}. 


The equivalence relation $T$ induces a function $m\colon\mathbb{S}^2\rightarrow \mathbb{Z}^+$ from the $2$-sphere to the natural numbers which counts the cardinality of the equivalence class to which each point belongs ($m$ stands for \emph{multiplicity}). The condition that $f$ is $C^1$ has several immediate implications.
\begin{itemize}
\item[1.] An easy compactness argument shows that the range of $f$ is in fact finite.
\item[2.] $m$ is upper semi-continuous\footnote{That is, $m^{-1}((-\infty, a))$ is open in $\mathbb{S}^2$.}. This means that multiplicity can jump down but not up. 
\newline 
\end{itemize}

If  $f$ is reasonable enough so that $X$ is a finite 2-complex, then we can see clearly by sheet counting that the multiplicities $m$ obey the additivity rule\footnote{In other words, the multiplicity function behaves locally as an invariant transverse measure taking values in $2\mathbb{Z}^+$.} at switches familiar from traintracks in one lower dimension.  
Let $d$ be the largest common divisor of $m (\mathbb{S}^2)$. Then the function $(m/d)$ takes integer values, locally obeys additivity and assumes an \emph{odd} value on at least some sheet of $X$ (and therefore some sheet of $Y$). 

The advantage of having an immersion of a branched manifold (that is also a finite complex) is that one can, thinking of it as a homology cycle, divide the multiplicities of the sheets by a common divisor and still get a mod-$2$ homology cycle. This is a difference between smooth immersions and PL immersions (e.g. the classical Bing house is the image of a from $\mathbb{S}^2$ that is generically 2-to-1, but when the sheets of the classical Bing house are labeled without multiplicity, each by 1, they do not constitute a cycle with $\mathbb{Z}_2$, or any other, coefficient system). 

In the case when $X$ is a finite 2-complex, with the weighting $(m/d)$, $X$ and its immersed image $Y$ may be interpreted as a homology cycle with $\mathbb{Z}_2$ coefficients. 
Note that we cannot replace ``$\mathbb{Z}_2$ coefficients" by ``$\mathbb{Z}$ coefficients" since following long loops in $X$ (or $Y$) can reverse a local sheet orientation. In fact, the example of Boy's surface warns us not to expect global additivity with $\mathbb{Z}$ coefficients.

However, $\mathbb{Z}_2$ coefficients are perfectly adequate (in the classical case) to complete the proof for the case when $X$ is a finite 2-complex. Consider a short arc $\alpha$ transverse to a sheet of $Y$ at a point where $h \colon X \looparrowright Y$ is locally $1-1$. Such points are open dense in $Y$ by the $C^1$-condition. Now use the property of any spine of $\mathbb{D}^3$ that its complement has the homotopy type of $\mathbb{S}^2$ (and in particular is connected) to connect the ends of $\alpha$ up by an arc $\beta$ in $\mathbb{D}^3\setminus Y$ to form a simple closed curve $\gamma = \alpha \cup \beta$ meeting the $\mathbb{Z}_2$-cycle $Y$ in a single point. This is a homological contradiction.
\newline

Therefore, all we need do now is to regularize our hypothetical $C^1$-Bing house from a possibly ``wild" Bing house to a finite simplicial $2$-complex whose homological properties are easier to study. If this paper were being written in the 1950’s, it is possible that the authors and auditors alike would know enough about the many variants of (co)homology theories available for the study wild spaces to finish the proof by merely identifying the appropriate replacement for $\mathbb{Z}_2$ singular (or simplicial) homology used for the classical case above. Lacking such scholarship we instead use a simple geometric trick to regularize Y without disturbing its dual loop $\gamma$. More precisely, what is left to prove is the following lemma.

\begin{lemmaB}[Sheet compression]\label{sheet flattening}
 Any $C^1$-immersion $f \colon \mathbb{S}^2 \looparrowright \mathbb{D}^3$ parametrizing a Bing spine is regularly homotopic to a parametrization $f'\colon \mathbb{S}^2 \looparrowright \mathbb{D}^3$ of another Bing spine $Y$ which is an immersion of a $C^1$ branched 2-manifold $X$ that is also a finite 2-complex.
\end{lemmaB}

\begin{proof}
The idea is first to strictly increase the equivalence relation $T$ to $T'$ by building regular homotopy $f_t$ from $f$ to $f'$ so as to ensure that all the nonempty pre-images $m_{f'}^{-1}(n)$, for $n\in 2\mathbb{Z}^+$, are codimension zero PL submanifolds of $\mathbb{S}^2$. After a routine perturbation to make the branching loci (the boundaries of these codimension zero PL submanifolds) of $f'$ mutually transverse, $X$ will be a classical branched surface and $Y$ will be its $C^1$-immersed image. 

By upper-semi-continuity, the set $D_{2n} \subset \mathbb{S}^2$ of highest multiplicity $2n$ (the maximum of $m$) is closed. We now press together the sheets of $f$ near $D_{2n}$ to enlarge $D_{2n}$ to $D'_{2n}$, a codimension zero PL submanifold. Now, $D'_{2n}$ can be constructed in an arbitrarily small $\varepsilon$-neighborhood $N_\varepsilon(D_{2n})$, for some $\varepsilon >0$, i.e. $D'_{2n}\subset N_\varepsilon(D_{2n})$. By picking $\varepsilon$ sufficiently small we can avoid \emph{coincidental} upward jumps in $m$ that would occur if the expanded multiplicity-$2n$ sheets happened to become tangent to other independent sheets. Avoiding this is possible because in a T4-space disjoint closed set can be engulfed in disjoint open sets.

One may proceed inductively downward, next enlarging $D_{2n-2}$ to a  codimension zero PL submanifold $D'_{2n-2}$, using a relative version of the previous argument which leaves $D'_{2n}$ unchanged. This inductive step and the final perturbation mentioned above complete the proof of Lemma B.
\end{proof}

By choosing the short arc $\alpha$ in the interior of its ($\text{odd}\times d$)-weight sheet, all the choices of $\varepsilon'$s in the proof of Lemma B can be made small enough that the successive deformation never touch a small neighborhood of the closed loop $\gamma$. Now we see that $\gamma$ is also geometrically dual to the (singular) $\mathbb{Z}_2$-cycle carried by $Y’$, which is homologically impossible. We have completed the proof of Theorem A. 
\end{proof}

A careful reading of this proof reveals that it uses no low dimensional features. Branched 2-manifolds are merely branched $n$-manifolds when $n=2$ (see \cite{Williams} for an exposition of branched $n$-manifolds). Neither does the proof use the fact that $f$ is a map from $\mathbb{S}^2$ but not any other surface.  
The argument readily generalizes to show: 

\begin{theoremA'}
For any $n\geq 2$, there is never a $C^1$-immersion $M^{n-1} \looparrowright \mathbb{D}^n$ from a closed manifold $M^{n-1}$ parametrizing a Bing house $Y$ in $\mathbb{D}^n$.
\end{theoremA'}

\begin{remark}
The above argument also shows that any smooth (or $C^1$) immersion from a closed manifold $M^n\rightarrow \mathbb{R}^{n+1}$ must have disconnected complement. So the homework problem mentioned at the beginning of the paper is correct if the embedding is replaced by an immersion as long as it is smooth (or $C^1$). The readers who has read up to this point could safely assign it to their students as a homework problem, even with the hypothesis on orientability of $M$ dropped.
\end{remark}

\bibliography{Reference}
\bibliographystyle{amsplain}

\end{document}